\documentclass{article}
\usepackage{amsmath}
\usepackage{amsfonts}
\usepackage{amsthm}
\usepackage{amssymb}
\usepackage{color}

\newtheorem{theorem}{Theorem}[section]
\newtheorem{lemma}[theorem]{Lemma}

\newtheorem{observation}[theorem]{Observation}

\newtheorem{corollary}[theorem]{Corollary}

\theoremstyle{definition}
\newtheorem{definition}[theorem]{Definition}
\theoremstyle{remark}
\newtheorem{remark}[theorem]{Remark}

\newcommand\remove[1]{}

\newcommand{\conv}{{\rm conv}\hskip0.02cm}

\def\f2{\mathbb{F}_2}

\newcommand{\1}{\mathbf{1}}

\newcommand{\ep}{\varepsilon}

\begin{document}

\title{On embeddings of locally finite metric spaces into $\ell_p$}

\author{Sofiya~Ostrovska\\
\\
Department of Mathematics\\ Atilim University \\ 06830 Incek,
Ankara, TURKEY \\
\textit{E-mail address}:
\texttt{sofia.ostrovska@atilim.edu.tr}\\
\\
and \\
\\
Mikhail~I.~Ostrovskii\footnote{corresponding author}\\
\\
Department of Mathematics and Computer Science\\
St. John's University\\
8000 Utopia Parkway\\
Queens, NY 11439, USA \\
\textit{E-mail address}: \texttt{ostrovsm@stjohns.edu}\\
Phone: 718-990-2469\\
Fax: 718-990-1650}

\date{\today}
\maketitle

\begin{abstract}

It is known that if finite subsets of a locally finite metric
space $M$ admit $C$-bilipschitz embeddings into $\ell_p$ $(1\le
p\le \infty)$, then for every $\ep>0$, the space $M$ admits a
$(C+\ep)$-bilipschitz embedding into $\ell_p$. The goal of this
paper is to show that for $p\ne 2,\infty$ this result is sharp in
the sense that  $\ep$ cannot be dropped out of  its statement.

\end{abstract}

\medskip


\begin{large}

\section{Introduction and Statement of Results}

During the last decades, the study of bilipschitz embeddings of
metric spaces into Banach spaces has become a field of intensive
research with a great number of applications. The latter are not
restricted to the area of Functional Analysis, but also include
Graph Theory, Group Theory, and Computer Science. We refer to
\cite{Lin02, Mat02, Nao10, Ost13, WS11}.  This work is focused on
the study of relations between the embeddability into $\ell_p$ of an infinite
metric space and its finite pieces. Let us recollect some needed
notions.

\begin{definition} A metric space is called {\it locally finite}
if each ball of finite radius in it has finite cardinality.
\end{definition}

\begin{definition}\label{D:LipBilip} Let $(A,d_A)$ and $(Y,d_Y)$ be metric spaces. Given,  $1\le C<\infty$,  a map $f:A\to Y$, is called a
{\it $C$-bilipschitz embedding} if there exists $r>0$ such that
\begin{equation}\label{E:MapDist}\forall u,v\in A\quad rd_A(u,v)\le
d_Y(f(u),f(v))\le rCd_A(u,v).\end{equation} A  map $f$ is a {\it
bilipschitz embedding} if it is $C$-bilipschitz for some $1\le
C<\infty$. The smallest constant $C$ for which there exists $r>0$
such that \eqref{E:MapDist} is satisfied, is called the {\it
distortion} of $f$.
\end{definition}

Unexplained terminology can be found in \cite{LT77,Ost13}.\medskip

It has been known that the bilipschitz embeddability of locally
finite metric spaces into Banach spaces is finitely determined in
the sense described by  the following theorem.

\begin{theorem}[\cite{Ost12}]\label{T:bilip} Let $A$ be a locally finite metric space whose
finite subsets admit bilipschitz embeddings with uniformly bounded
distortions into a Banach space $X$. Then, $A$ also admits a
bilipschitz embedding into $X$.
\end{theorem}

Theorem \ref{T:bilip} has many predecessors, see
\cite{Bau07,Bau12,BL08,Ost06,Ost09}. Applications of this theorem
to the coarse embeddings important for Geometric Group Theory and
Topology are discussed in \cite{Ost12}.  To expand on the theme,
the argument of \cite{Ost12} yields a stronger result, namely the
one stated as Theorem \ref{T:FDwithD}. In order to formulate
Theorem \ref{T:FDwithD}, it is handy to employ the parameter
$D(X)$ of a Banach space $X$ introduced in \cite{OO18+}. Let us
recollect its definition. Given a Banach space $X$ and a real
number $\alpha\ge 1$, we write:

\begin{itemize}

\item  $D(X)\le\alpha$ if, for any locally finite metric space
$A$, all finite subsets of which admit bilipschitz embeddings into
$X$ with distortions $\le C$, the space $A$ itself admits a
bilipschitz embedding into $X$ with distortion $\le \alpha\cdot
C$;

\item  $D(X)=\alpha$ if $\alpha$ is the least number for which
$D(X)\le\alpha$;

\item  $D(X)=\alpha^+$ if,  for every $\ep>0$, the condition
$D(X)\le\alpha+\ep$ holds, while $D(X)\le\alpha$ does not;

\item $D(X)=\infty$ if $D(X)\le\alpha$ does not hold for
any $\alpha<\infty$.

\end{itemize}

In addition, we use inequalities like $D(X)<\alpha^+$ and
$D(X)<\alpha$ with the natural meanings, for example
$D(X)<\alpha^+$ indicates that either $D(X)=\beta$ for some
$\beta\le\alpha$ or $D(X)=\beta^+$ for some $\beta<\alpha$.

\begin{theorem}[\cite{Ost12}]\label{T:FDwithD} There exists an absolute constant $D\in[1,\infty)$, such
that for an arbitrary Banach space $X$ the inequality $D(X)\le D$
holds.
\end{theorem}

Recently, new estimates of the parameter $D(X)$ for some classes
of Banach spaces have been obtained in \cite{OO18+}. Recall that a
family of finite-dimensional Banach spaces $\{X_n\}_{n=1}^\infty$
is said to be {\it nested} if $X_n$ is a proper subspace of
$X_{n+1}$ for every $n\in\mathbb{N}$. For such families, an
 estimate for $D(X)$ from above is expressed by:

\begin{theorem}[{\cite[Theorem 1.9]{OO18+}}]\label{T:Above}  Let $1\le
p<\infty$. If $\{X_n\}_{n=1}^\infty$ is a nested family of
finite-dimensional Banach spaces, then
$\displaystyle{D\left(\left(\oplus_{n=1}^\infty
X_n\right)_p\right)\le 1^+}$.
\end{theorem}

The next assertion is an immediate consequence of Theorem
\ref{T:Above}:

\begin{corollary}[{\cite[Corollary 1.10]{OO18+}}]\label{C:Ell_p} If $1\le p<\infty$, then
$D(\ell_p)\le 1^+$.
\end{corollary}

It should be mentioned that the case where  $p=\infty$ was
discarded because the classical result of Fr\'echet \cite{Fre10}
implies that $D(\ell_\infty)=1$. Observe also that it is a
well-known fact that $D(\ell_2)=1$. Although the paper
\cite{OO18+} contains some estimates for $D(X)$ from below, the
following question was left open: whether $D(\ell_p)=1^+$ or
$D(\ell_p)=1$ for $1\le p<\infty$, $p\ne 2$?\medskip

The main goal of this paper is to complete the picture by proving
that $D(\ell_p)\ge 1^+$ if $p\in[1,\infty)$, $p\ne 2$. See
Theorem \ref{T:ell1>1} and Corollary \ref{C:ellp>1}. It is worth
pointing out that our proofs for the cases $p=1$ and $p>1$ are
different from each other.
\medskip

Recall that a Banach space is called {\it strictly convex} if its
unit sphere does not contain line segments. In the present work,
it is shown that $D(X)>1$ for a large class of strictly convex
Banach spaces $X$ implying that $D(X)=1^+$ for all strictly convex
Banach spaces satisfying the assumption of Theorem \ref{T:Above}.
To be more specific, the following statement will be proved (see
Section \ref{S:p>1}):

\begin{theorem}\label{T:StrConv} Let $X$ be a strictly convex Banach space such that all finite subsets of $\ell_2$
admit isometric embeddings into $X$, but $\ell_2$ itself does not
admit an isomorphic embedding into $X$. Then $D(X)>1$.
\end{theorem}

With the help of Theorem \ref{T:StrConv}, one derives:

\begin{corollary}\label{C:ellp>1} Let $p\in(1,\infty)$, $p\ne 2$. Then every strictly convex Banach space of
the form $X=\left(\left(\oplus_{n=1}^\infty X_n\right)_p\right)$,
where $\{X_n\}_{n=1}^\infty$ is a nested sequence of
finite-dimensional Banach spaces satisfies $D(X)>1$.
\end{corollary}

Combining Theorem \ref{T:Above} and Corollary \ref{C:ellp>1} one
obtains:

\begin{corollary}\label{C:ellp=1+} Let $p\in(1,\infty)$, $p\ne 2$, and let $\{X_n\}_{n=1}^\infty$ be a nested family of finite
dimensional strictly convex Banach spaces. Then, the space
$\displaystyle{X=\left(\oplus_{n=1}^\infty X_n\right)_p}$
satisfies $D(X)= 1^+$. The equality $D(\ell_p)=1^+$ for
$p\in(1,\infty),~p\ne 2$, follows as a special case of this
result.
\end{corollary}

The case $p=1$ is quite different because $\ell_1$ is not strictly
convex. This case is examined in Section \ref{S:p=1}, where we
prove:

\begin{theorem}\label{T:ell1>1} $D(\ell_1)>1$.
\end{theorem}

Juxtaposing this outcome with Theorem \ref{T:Above}, we reach:

\begin{corollary}\label{C:ell1=1+} $D(\ell_1)=1^+$.
\end{corollary}

\begin{remark} It should be mentioned that the above results are
not the first known ones claiming $D(X)>1$. Before now, results of
this kind were obtained in \cite[Theorem 2.9]{KL08} and
\cite[Theorem 1.12]{OO18+} for some other Banach spaces and their
classes.
\end{remark}

\section{Proof of Theorem \ref{T:StrConv}}\label{S:p>1}

Prior to presenting the proof of Theorem \ref{T:StrConv}, let us
provide some auxiliary information. By developing the notion of a
linear triple \cite[p.~56]{Blu53}, we introduce the following:

\begin{definition}\label{D:LinearnTup} A collection $r=\{r_i\}_{i=1}^n$ of points in a metric space $(A,d_A)$ is called a {\it linear
tuple} if the sequence $\{d_A(r_i,r_1)\}_{i=1}^n$ is strictly
increasing and, for $1\le i<j<k\le n$, the equality below holds:
\begin{equation}\label{E:TriangleEq}
d_A(r_i,r_k)=d_A(r_i,r_j)+d_A(r_j,r_k).
\end{equation}
\end{definition}

It is not difficult to see that, for strictly convex Banach
spaces, the following statement is valid.

\begin{observation}\label{O:ImLinTup} An isometric image of a linear tuple $r=\{r_i\}_{i=1}^n$ in a strictly convex Banach space is contained in the line segment joining
the images of $r_1$ and $r_n$.
\end{observation}

For the sequel, the next fact  is needed (by $B_Z$ we denote the
unit ball of a Banach space $Z$):

\begin{lemma}\label{L:FromNet} Let $Z$ be a finite-dimensional Banach space and $F$ be a Banach space. Then, for each $\ep>0,$ there exists
$\delta=\delta(\ep,Z,F)>0$ such that if a $\delta$-net in $B_Z$
admits an isometric embedding into  $F$, then $F$ contains a
subspace whose Banach-Mazur distance to $Z$ does not exceed
$(1+\ep)$.
\end{lemma}

 Lemma \ref{L:FromNet} is an immediate consequence
of Bourgain's discretization theorem \cite{Bou87}. It should be
emphasized that this theorem provides a  much stronger claim
because Bourgain found an explicit estimate for $\delta$ as a
function of $\ep$ and the dimension of $Z$; besides in Bourgain's
theorem, the distortion of embedding of $Z$ is estimated in terms
of distortion of embedding of a $\delta$-net of $B_Z$. See
\cite{Beg99,GNS12} for simplifications of Bourgain's proof, see
also its presentation in \cite[Section 9.2]{Ost13}. Meanwhile, the
existence of $\delta(\ep,Z,F)$ can be derived from earlier results
of Ribe \cite{Rib76} and Heinrich and Mankiewicz \cite{HM82}, see
\cite[p.~818]{GNS12}.

\begin{proof}[Proof of Theorem \ref{T:StrConv}]
Denote the unit vector basis of $\ell_2$ by
$\{e_i\}_{i=1}^\infty$. Our intention  is to find a locally finite
subset $M$ of $\ell_2$ in such a way  that:
\medskip

{\bf (A)} $M$ contains a $\delta(\frac 1n, \ell_2^n, X)$-net $M_n$ of a shifted unit ball $y_n+B_{\ell_2^n}$, $n\in \mathbb{N}$.
\medskip

{\bf (B)} There exists a sequence $\{\alpha_i\}_{i=1}^\infty$ of
positive numbers, such that if $T:M\to X$ is an isometry
satisfying $T(0)=0$, then the image of $T(M_n)$ is contained in
the linear span of $\{T(\alpha_1e_1),\dots,T(\alpha_n e_n)\}$.
\medskip

The existence of such a set $M$ will prove Theorem \ref{T:StrConv}
because, by the assumption of the theorem, finite subsets of $M$
admit isometric embeddings into $X$. On the other hand, $M$ itself
does not admit an isometric embedding into $X$. In fact, such an
embedding $T$ could be assumed to satisfy $T(0)=0$. Therefore
conditions {\bf (A)} and {\bf (B)}, combined with Lemma
\ref{L:FromNet}, would imply that the Banach-Mazur distance
between the linear span of $\{T(\alpha_1e_1),\dots,T(\alpha_n
e_n)\}$ and $\ell_2^n$ does not exceed $\left(1+\frac1n\right)$.
It is well known (see \cite{Joi66}) that, in this case, $X$
contains a subspace isomorphic to $\ell_2$, which is a
contradiction.
\medskip

We let
\[M=\left(\bigcup_{n=1}^\infty
M_n\right)\bigcup \{0\},\] where $M_n$ are finite sets constructed
in the way described hereinafter.

Denote by $R_i, i\in \mathbb{N}$,  the positive rays generated by
$e_i$, that is, $R_i=\{\alpha e_i:~\alpha\ge 0\}$. Let $M_1$ be
the $\delta(1, \ell_2^1, X)$-net in the line segment $[0,2e_1]$,
where we assume that $M_1$ includes $e_1$. It is clear that $M_1$
satisfies \textbf{(A)}.

For $n>1$ sets $\{M_n\}_{n=1}^\infty$ will be constructed
inductively. Suppose that we have already created
$M_1,\dots,M_{n-1}$. To construct $M_n$, we pick points $s^n_i\in
R_i$, $1\leq i\leq n$, and one more point, $s^n_{n+1}\in R_n$ - so
that $R_n$ contains both $s^n_n$ and $s^n_{n+1}$ - in such a way
that $\conv(\{s^n_i\}_{i=1}^{n+1})$ is at distance at least $n$
from the origin, and $\conv(\{s^n_i\}_{i=1}^{n+1})$ contains a
shift $y_n+B_{\ell_2^n}$ of the unit ball (for some $y_n$). This
is clearly possible. Next, we select a $\delta(\frac1n, \ell_2^n,
X)$-net $\mathcal{N}_n$ in this shifted unit ball
$y_n+B_{\ell_2^n}$  and include it in $M_n$ together with
$\{s^n_i\}_{i=1}^{n+1}$. At this point, it is evident that
condition {\bf (A)} is satisfied.

To ensure that condition {\bf (B)} is also satisfied - as it will
be seen later - we add,  for each element $z\in \mathcal{N}_n$,
finitely many additional elements of
$\conv(\{s^n_i\}_{i=1}^{n+1})$ to $M_n$ according to the procedure
suggested below:

\begin{itemize}

\item If $z\in \{s^n_i\}_{i=1}^{n+1}$, there is nothing to
include. If $z\notin \{s^n_i\}_{i=1}^{n+1}$, we find and include
in $M_n$ an element $w_1(z)$ in a convex hull of an $n$-element
subset $W_1(z)$ of $\{s^n_i\}_{i=1}^{n+1}$ with $z$ being on the
line segment joining $w_1(z)$ and $s_i^n\in
\left(\{s^n_i\}_{i=1}^{n+1}\backslash W_1(z)\right)$.

\item If $w_1(z)\in \{s^n_i\}_{i=1}^{n+1}$, there is nothing else
to include. If $w_1(z)\notin \{s^n_i\}_{i=1}^{n+1}$, we find and
include in $M_n$ an element $w_2(z)$ in a convex hull of an
$(n-1)$-element subset $W_2(z)$ of $\{s^n_i\}_{i=1}^{n+1}$ such
that $w_1(z)$ is on the line segment joining $w_2(z)$ and
$s_i^n\in \left(\{s^n_i\}_{i=1}^{n+1}\backslash W_2(z)\right)$.

\item We continue in an obvious way.

\item If we do not terminate the process in one of the previous
steps, we arrive at the situation when $w_n(z)$ is in a convex
hull of a $2$-element subset of $\{s^n_i\}_{i=1}^{n+1}$, and hence
it is on some line segment of the  form $[s^n_i,s^n_j]$. At this
point we stop.

\end{itemize}

It has already been stated   that condition {\bf (A)} is satisfied
for $M$. Now, let us verify condition {\bf (B)}. To do this, it
suffices to prove that, for each isometry $T:(M_n\cup\{0\})\to X$
satisfying $T(0)=0$, the image $T(M_n)$ is contained in the linear
span of $\{Ts^n_1,\dots,Ts^n_n\}$. This condition looks slightly
different from the one in {\bf (B)}. However, defining
$\{\alpha_i\}_{i=1}^\infty$ by $\alpha_1=1$ and
$\alpha_ie_i=s^i_i$ one can see that in essence the conditions are
equivalent because, by Observation \ref{O:ImLinTup}, the images
$\{Ts^n_i\}_{n=i}^\infty$ are multiples of each other.
\medskip

To show that $T(M_n)$ is contained in the linear span $L$ of
$\{Ts^n_1,\dots,Ts^n_n\}$, the procedure outlined underneath is
applied, where in each step Observation \ref{O:ImLinTup} is used.

\begin{itemize}

\item Since $0, s^n_n$, and $s^n_{n+1}$ form a linear tuple, and
$T(0)=0$, we have $T(s^n_{n+1})\in L$.

\item Whenever  $w_n(z)$ is defined, one has $Tw_n(z)\in L$
because $w_n(z)\in [s^n_i,s^n_j]$.

\item Likewise, for each $z$ such that $w_{n-1}(z)$ is defined,
one obtains $Tw_{n-1}(z)\in L$ since  $w_{n-1}(z)$ is in the line
segment joining $w_n(z)$ and one of $s^n_i$.

\item We proceed in a straightforward  way till we get $Tz\in L$.

\end{itemize}

In addition, it is easy to see that the assumption that
$\conv(\{s^n_i\}_{i=1}^{n+1})$ is at distance at least $n$ from
the origin together with the fact that each set $M_n$ is finite
and is contained in  $\conv(\{s^n_i\}_{i=1}^{n+1})$, implies that
the set $\cup_{n=1}^\infty M_n$ is locally finite.
\end{proof}

\begin{proof}[Proof of Corollary \ref{C:ellp>1}] To check that this $X$ satisfies the conditions of
Theorem \ref{T:StrConv}  the two well-known facts come in
handy:\medskip

\noindent{\bf (1)} Each finite subset of $L_p[0,1]$ admits an
isometric embedding into $\ell_p$, see \cite{Bal90}.

\noindent{\bf (2)} The space $L_p[0,1]$ contains a subspace
isometric to $\ell_2$, see \cite[p.~16]{JL01}.
\medskip

Combining {\bf (1)} and {\bf (2)} we conclude  that all finite
subsets of $\ell_2$ are isometric to subsets of $\ell_p$, and,
thence, to subsets of $X$. On the other hand, it is known that
each infinite-dimensional subspace of $X$ contains a subspace
isomorphic to $\ell_p$ (this can be done using a slight variation
of the argument used to prove \cite[Proposition 2.a.2]{LT77}), and
as such it is not isomorphic to $\ell_2$.
\end{proof}

\begin{remark} The first part of the proof of Theorem
\ref{T:StrConv} demonstrates that its statement can be
strengthened by replacing the condition ``$\ell_2$ does not admit
an isomorphic embedding into $X$'' by ``there is $\alpha>1$ such
that $X$ does not contain a subspace whose Banach-Mazur distance
to $\ell_2$ does not exceed $\alpha$''. It is known \cite{OS94}
that the latter condition is strictly weaker. In addition, it is
not difficult to see that although Joichi did not formally state
the pertinent modification of the main result of \cite{Joi66}, it
arises from the proof.
\end{remark}

\section{The case of $\ell_1$}\label{S:p=1}

\begin{proof}[Proof of Theorem \ref{T:ell1>1}]
Recall \cite{Bal90} that each finite subset of
$L_1(-\infty,\infty)$  admits an isometric embedding into
$\ell_1$. To prove Theorem \ref{T:ell1>1} we construct in
$L_1(-\infty,\infty)$ a locally finite metric space $M$ such that
its isometric embeddability into $\ell_1$ would imply that
$\ell_1$ contains a unit vector $x$ which, for every
$n\in\mathbb{N}$,  can be represented as a sum of $2^n$ vectors
with pairwise disjoint supports and of norm $2^{-n}$ each. This
leads to a contradiction: consider the maximal in absolute value
coordinate of the vector $x$, let it be $\alpha$. If, for some
$n\in\mathbb{N},\;|\alpha|>2^{-n}$, it is clearly impossible to
partition vector into $2^n$ vectors of norm $2^{-n}$ each with
pairwise disjoint supports.
\medskip

The starting point of the construction is the fact that the
indicator function $\1_{(0,1]}$ has, for each $n\in\mathbb{N}$, a
representation as a sum of $2^n$ pairwise disjoint vectors of norm
$2^{-n}$. To be specific, we adopt the writing:

\begin{itemize}

\item $\1_{(0,1]}=d_0+d_1$, where $d_0=\1_{(0,\frac12]},
d_1=\1_{(\frac12,1]}$

\item $\1_{(0,1]}=d_{00}+d_{01}+d_{10}+d_{11}$, where
$d_{00}=\1_{(0,\frac14]}, d_{01}=\1_{(\frac14,\frac12]},
d_{10}=\1_{(\frac12,\frac34]}, d_{11}=\1_{(\frac34,1]}$

\item We carry on  in an obvious way.

\end{itemize}

In the sequel,  the following notation will be employed: let
$d=\1_{(0,1]}$ and denote the functions introduced above by
$d_\sigma$, where $\sigma$ is a finite string of $0$'s and $1$'s.
Denote by $\ell(\sigma)$ the length of the string $\sigma$. For
each $\sigma=\{\sigma_i\}_{i=1}^{\ell(\sigma)}$, the subinterval
$I(\sigma)$ of $(0,1]$ is defined  by:
\[I(\sigma)=\left(\sum_{i=1}^{\ell(\sigma)}\sigma_i2^{-i},~2^{-\ell(\sigma)}+\sum_{i=1}^{\ell(\sigma)}\sigma_i2^{-i}\right].
\]

With this notation $d_\sigma=\1_{I(\sigma)}$ and the mentioned
above representation of $\1_{(0,1]}$ as a sum of $2^n$ terms can
be written as:
\[d=\sum_{\sigma,~ \ell(\sigma)=n}d_\sigma,\]
where the summands are disjointly supported. Now, denote by
$\mathcal{T}$ the set of all finite strings of $0$'s and $1$'s. It
is obvious  that $\{d_\sigma\}_{\sigma\in \mathcal{T}}$ is not a
locally finite set. Nonetheless, we can add to $\{d_\sigma\}$
pairwise disjoint functions in such a way that  a locally finite
subset of $L_1(-\infty,\infty)$ will be obtained, and the
existence of an isometric embedding of this set into $\ell_1$
would imply the existence in $\ell_1$ of a vector $x$ with the
properties described at the beginning of the proof.

First, opt for  an injective map $\Psi$ from the collection of all
finite strings of $0$'s and $1$'s into
$\mathbb{Z}\backslash\{0\}$.
\medskip

Now, we consider the locally finite set $M$ satisfying the
conditions: It contains both functions $d$ and $0$, and, in
addition, it includes  all sums
$f_\sigma:=d_{\sigma}+\ell(\sigma)\cdot
\1_{(\Psi(\sigma),\Psi(\sigma)+1]}$, where $\sigma\in\mathcal{T}$.

Since $\ell(\sigma)$ is less than any fixed constant only for
finitely many strings $\sigma$, this set is a locally finite
subset of $L_1(-\infty,\infty)$. It has to  be shown that
isometric embeddability of this set into $\ell_1$ implies the
existence in $\ell_1$ of a vector $x$ with the properties stated
in the first paragraph of the proof, thus resulting in a
contradiction.
\medskip

Indeed, if there is an isometric embedding of $M$ into $\ell_1$,
then there is an isometric embedding $F$ which maps $0$ to $0$ and
- as it will be
 proved - in such a case $x=Fd$ is the desired
vector. More elaborately put, the existence of such an isometric
embedding implies  that there exist vectors
$\{x_\sigma\}_{\sigma\in\mathcal{T}}$ so that, for each
$n\in\mathbb{N}$,  the vectors $\{x_\sigma\}_{\ell(\sigma)=n}$ are
disjointly supported, have norm $2^{-n}$, and
\[x=Fd=\sum_{\sigma,~\ell(\sigma)=n}x_\sigma.\]
\medskip

Each element $a=\sum_{i=1}^\infty a_ie_i$ of $\ell_1$ can be
considered as a possibly infinite union of intervals in the
coordinate plane which join  $(i,0)$ and $(i,a_i)$. The total length of
all intervals is equal to $||a||$.
\medskip

The proposed  construction guarantees that if $\ell(\sigma)=n$,
then $||f_\sigma-d||=||f_\sigma||+||d||-2\cdot 2^{-n}$. Since $F$
is an isometry, $F(0)=0$ and $Fd=x$, this implies
$||Ff_\sigma-x||=||Ff_\sigma||+||x||-2\cdot 2^{-n}$. Consequently,
the total length of intersections of the intervals corresponding
to $x$ and to $Ff_\sigma$ is $2^{-n}$ for $\sigma\in \{0,1\}$.

On the other hand, if $\sigma\ne\tau$ and
$\ell(\sigma)=\ell(\tau)=n$,  the functions $f_\sigma$ and
$f_\tau$ are disjointly supported and, therefore,
$||f_\sigma-f_\tau||=||f_\tau-0||+||f_\sigma-0||$. As a result,
$||Ff_\sigma-Ff_\tau||=||Ff_\sigma||+||Ff_\tau||$. This means that
the intersections of the intervals corresponding to $Ff_\sigma$
and $Ff_\tau$ have total length $0$. It does not immediately imply
that  vectors $Ff_\sigma$ and $Ff_\tau$ are disjointly supported:
one can imagine, for example, that $Ff_\sigma$ contains the
interval joining $(i,0)$ and $(i,\frac14)$ and $Ff_\tau$ contains
the interval joining $(i,0)$ and $(i,-\frac14)$.

Let us define the vector $x_\sigma$ for $\sigma$ satisfying
$\ell(\sigma)=n$ as a vector for which the corresponding intervals
are intersections of the intervals corresponding to $x$ and to
$Ff_\sigma$. The previous paragraphs imply that $x_\sigma$ and
$x_\tau$ satisfy $||x_\sigma||=||x_\tau||=2^{-n}$ and have
disjoint supports when $\ell(\sigma)=\ell(\tau)=n$ and
$\sigma\ne\tau$ (for the latter we use the fact that the interval
corresponding to $x$ at $i$ can have `positive' or `negative'
part, but not both).

Finally, let $s=\sum_{\sigma,~\ell(\sigma)=n}x_\sigma$. With the
preceding arguments, we conclude that $||s||=1$ and $|s_i|\le
|x_i|$ for each $i\in\mathbb{N}$. Thus, $s=x$, and the desired
decomposition of $x$ is completed.
\end{proof}

\subsection*{Acknowledgement}

The second-named author gratefully acknowledges the support by
National Science Foundation grant NSF DMS-1700176.

\end{large}

\renewcommand{\refname}{

\end{document}